\documentclass[12pt]{amsart}

\usepackage{amssymb}
\usepackage{bm}
\usepackage{graphicx}
\usepackage{psfrag}
\usepackage{color}
\usepackage{url}

\usepackage{algpseudocode}

\usepackage{mathtools}

\usepackage{xy}
\input xy
\xyoption{all}

\newcommand{\excise}[1]{}

\newtheorem{thm}{Theorem}[section]

\theoremstyle{definition}

\newtheorem{example}[thm]{Example}
\newtheorem{remark}[thm]{Remark}
\newtheorem{defn}[thm]{Definition}

\numberwithin{equation}{section}

\newcommand{\ring}[1]{\ensuremath{\mathbb{#1}}}

\renewcommand\>{\rangle}
\newcommand\<{\langle}

\newcommand\NN{\ring{N}}

\newcommand\ZZ{\ring{Z}}

\renewcommand\aa{{\mathbf a}}

\DeclareMathOperator\lcm{lcm} 

\DeclareMathOperator\Betti{Betti} 

\begin{document}

\mbox{}
\title[Realizable sets of catenary degrees of numerical monoids]{Realizable sets of catenary degrees\\of numerical monoids}

\author{Christopher O'Neill}
\address{Mathematics\\University of California, Davis\\One Shields Ave\\Davis, CA 95616}
\email{coneill@math.ucdavis.edu}

\author{Roberto Pelayo}
\address{Mathematics Department\\University of Hawai`i at Hilo\\Hilo, HI 96720}
\email{robertop@hawaii.edu}

\date{\today}

\begin{abstract}

The catenary degree is an invariant  that measures the distance between factorizations of elements within an atomic monoid.  In this paper, we classify which finite subsets of $\ZZ_{\ge 0}$ occur as the set of catenary degrees of a numerical monoid (i.e., a co-finite, additive submonoid of $\ZZ_{\ge 0}$).  In particular, we show that, with one exception, every finite subset of $\ZZ_{\ge 0}$ that can possibly occur as the set of catenary degrees of some atomic monoid is actually achieved by a numerical monoid.


\end{abstract}

\maketitle


\section{Introduction} \label{sec:intro}

Nonunique factorization theory aims to classify and quantify the failure of elements of a cancellative commutative monoid $M$ to factor uniquely into irreducibles~\cite{nonuniq}.  This~is often achieved using arithmetic quantities called factoriation invariants.  We consider here the catenary degree invariant (Definition~\ref{d:catenarydegree}), which is a non-negative integer $\mathsf c(m)$ measuring the distance between the irreducible factorizations of an element $m \in M$.  

Many problems in factorization theory involve characterizing which possible values an invariant can take, given some minimal assumptions on the underlying monoid~$M$.  Solutions to such problems often take the form of a ``realization'' result in which a family of monoids achieving all possible values is identified.  
One such problem of recent interest concerns the possible values of the delta set invariant $\Delta(M)$, comprised~of successive differences in lengths of factorizations of elements.  It can be easily shown that $\min\Delta(M) = \gcd\Delta(M)$ for any atomic monoid \cite{nonuniq}, but no further restrictions were known.  The so-called ``delta set realization problem''~\cite{deltarealizationnumerical,deltadim3,deltadim3b} was recently solved by Geroldinger and Schmidt~\cite{deltarsetealization} by identifying a family of finitely generated Krull monoids whose delta sets achieve every finite set $D \subset \ZZ_{\ge 1}$ satisfying $\min(D) = \gcd(D)$.  

In a similar vein, the catenary degree realization problem \cite[Problem~4.1]{mincatdeg} asks which finite sets can occur as the set $\mathsf C(M)$ of catenary degrees achieved by elements of some atomic monoid $M$.  A~recent paper of Fan and Geroldinger~\cite{catenarycartesian} proves that, with minimal assumptions, the product $M = M_1 \times M_2$ of two monoids $M_1$ and $M_2$ has set of catenary degrees $\mathsf C(M) = \mathsf C(M_1) \cup  \mathsf C(M_2)$.  While this does provide a complete solution to the catenary degree realization problem, the brevity of the proof raises the question: is such a result is possible without appealing to Cartesian products?


In this paper, we prove that with only a single exception, any finite set that can occur as the set of catenary degrees of some atomic monoid occurs as the set of catenary degrees of a numerical monoid (Theorem~\ref{t:realization}).  Since numerical monoids are submonoids of a reduced rank one free monoid (namely, $\ZZ_{\ge 0}$), they never contain a Cartesian product of two nontrivial monoids.  

Our method of constructing numerical monoids with prescribed sets of catenary degrees involves carefully gluing smaller numerical monoids (Definition~\ref{d:gluing}) in such a way as to retain complete control of the catenary degree of every element (Theorem~\ref{t:gluing}).  The relationship between the catenary degree and gluings has been studied before~\cite{presentsacc}, though noteable subtleties arise; Remark~\ref{r:upperbound} clarifies some ambiguity in~\cite{presentsacc}, and is another primary contribution of this manuscript.

\section{Background}
\label{sec:background}

In what follows, $\NN = \ZZ_{\ge 0}$ denotes the set of non-negative integers.

\begin{defn}\label{d:numericalmonoid}
A \emph{numerical monoid} $S$ is an additive submonoid of $\NN$ with finite complement.  When we write $S = \<n_1, \ldots, n_k\>$, we assume $n_1 < \cdots < n_k$, and the chosen generators $n_1, \ldots, n_k$ are minimal with respect to set-theoretic inclusion.  These minimal generators are called \emph{irreducible} elements or \emph{atoms}.  
\end{defn}

\begin{defn}\label{d:factorizations}
Fix $n \in S = \<n_1, \ldots, n_k\>$.  A \emph{factorization} of $n$ is an expression $n = u_1 + \cdots + u_r$ of $n$ as a sum of atoms $u_1, \ldots, u_r$ of $S$.  
Write 
$$\mathsf Z_S(n) = \{(a_1, \ldots, a_k) : n = a_1n_1 + \cdots + a_kn_k\} \subset \ZZ_{\ge 0}^k$$
for the \emph{set of factorizations} of $n \in S$.  Given $\aa \in \mathsf Z_S(n)$, we denote by $|\aa|$ the number of irreducibles in the factorization $\aa$, that is, $|\aa| = a_1 + \cdots + a_k$.  
\end{defn}

We are now ready to define the catenary degree.  

\begin{defn}\label{d:catenarydegree}
Fix an element $n \in S = \<n_1, \ldots, n_k\>$ and factorizations $\aa, \aa' \in \mathsf Z_S(n)$.  The \emph{greatest common divisor of $\aa$ and $\aa'$} is given by 
$$\gcd(\aa,\aa') = (\min(a_1,b_1), \ldots, \min(a_k,b_k)) \in \ZZ_{\ge 0}^k,$$
and the \emph{distance between $\aa$ and $\aa'$} (or the \emph{weight of $(\aa,\aa')$}) is given by 
$$\mathsf d(\aa,\aa') = \max(|\aa - \gcd(\aa,\aa')|,|\aa' - \gcd(\aa,\aa')|).$$
Given $N \ge 0$, an \emph{$N$-chain from $\aa$ to $\aa'$} is a sequence $\aa = \aa_1, \aa_2, \ldots, \aa_k = \aa' \in \mathsf Z_S(n)$ of factorizations such that $\mathsf d(\aa_{i-1},\aa_i) \le N$ for all $i \le k$.  The \emph{catenary degree of~$n$}, denoted $\mathsf c_S(n)$, is the smallest $N \ge 0$ such that there exists an $N$-chain between any two factorizations of $n$.  The \emph{set of catenary degrees of $S$} is the set $\mathsf C(S) = \{\mathsf c(m) : m \in S\}$, and the \emph{catenary degree of $S$} is the supremum $\mathsf c(S) = \sup\mathsf C(S)$.  
\end{defn}

\begin{defn}\label{d:bettielements}
Fix a finitely generated monoid $S$.  For each nonzero $n \in S$, consider the graph $\nabla_n$ with vertex set $\mathsf Z(n)$ in which two vertices $\aa, \aa' \in \mathsf Z(n)$ share an edge if $\gcd(\aa,\aa') \ne 0$.  If $\nabla_n$ is not connected, then $n$ is called a \emph{Betti element} of $S$.  We write 
$$\Betti(S) = \{b \in S : \nabla_b \text{ is disconnected}\}$$
for the set of Betti elements of $S$.  
\end{defn}

We conclude this section by defining the gluing operation for numerical monoids.  For a more thorough introduction to gluing, see \cite[Chapter~8]{numerical}.  

\begin{defn}\label{d:gluing}
Fix numerical monoids $S_1$ and $S_2$, and positive integers $d_1$ and $d_2$.  The monoid $S = d_1S_1 + d_2S_2$ is a \emph{gluing of $S_1$ and $S_2$ by $d$} if $d = \lcm(d_1,d_2) \in S_1 \cap S_2$.  
\end{defn}

\section{Catenary degree of numerical monoid gluings}
\label{sec:gluing}

Theorem~\ref{t:gluing} provides the primary technical result used in Theorem~\ref{t:realization} to construct a numerical monoid with a prescribed set of catenary degrees, and is closely related to Theorem~\ref{t:pedro}, an ambiguously worded (Remark~\ref{r:upperbound}) result appearing in~\cite{presentsacc}.  

\begin{thm}[{\cite[Corollary~4]{presentsacc}}]\label{t:pedro}
If $S$ is a gluing of $S_1$ and $S_2$ by $d$, then 
$$\mathsf c(S) \le \max\{\mathsf c(S_1), \mathsf c(S_2), \mathsf c_S(d)\}$$
and equality holds if $\mathsf c(S_1)$ and $\mathsf c(S_2)$ are each computed by taking the maximum of $\mathsf c_S(n)$ over elements of $S_1$ and $S_2$, respectively.  
\end{thm}

\begin{remark}\label{r:upperbound}
The wording of Theorem~\ref{t:pedro} differs significantly from the original statement appearing in \cite{presentsacc}.  In particular, the original wording was ambiguous, and it was unclear that $\mathsf c(S_1)$ and $\mathsf c(S_2)$ should be computed in the (nonstandard) way specified above.  If these values are computed in the usual way, then the inequality in Theorem~\ref{t:pedro} can indeed be strict (see, for instance, Examples~\ref{e:gluingsmallc} and~\ref{e:gluingmoregens}); the unambiguous statement given in Theorem~\ref{t:pedro} is the result of discussions with the second author of~\cite{presentsacc} that took place after such a monoid $S$ was found.  
\end{remark}

Remark~\ref{r:upperbound} is an example of one of the many subtleties one encounters when working with the catenary degree.  It remains an interesting question to determine for which $S$ equality is achieved in Theorem~\ref{t:pedro}; Theorem~\ref{t:gluing} gives one such setting.  

\begin{thm}\label{t:gluing}
Suppose $S = \<n_1, \ldots, n_k\>$ is a numerical monoid.  Fix $c > \mathsf c(S)$, $b \in S$ nonzero with $\gcd(b,c) = 1$, and let $T = \<cn_1, \ldots, cn_k, b\>$.  Then 
$$\mathsf c_T(n) = \left\{\begin{array}{ll}
\mathsf c_S(n/c) & n - cb \notin T \\
c & n - cb \in T
\end{array}\right.$$
for any $n \in T$.  In particular, $\mathsf c(T) = c$.  
\end{thm}

\begin{proof}
Since $b \in S$ and $\gcd(b,c) = 1$, $T$ is a gluing of $S$ and $\NN$ by $cb$ by \cite[Theorem~8.2]{numerical}, meaning $\Betti(T) = c\Betti(S) \cup \{cb\}$ by \cite[Corollary~2]{presentsacc}.  As such, if $n - cb \notin T$, then $a_{k+1}$ is constant among all factorizations $\aa \in \mathsf Z_T(n)$, and thus $\mathsf c_T(n) = \mathsf c_S(n/c)$.  

Now, suppose $n - cb \in T$.  By \cite[Theorem~3.1]{catenarytamefingen}, 
$$\mathsf c(T) = \max\{\mathsf c_T(m) : m \in \Betti(T)\} = \mathsf c_T(cb) = c,$$
so it suffices to prove that $\mathsf c_T(n) \ge c$.  To this end, we will show that if $\aa, \aa' \in \mathsf Z_T(n)$ with $a_{k+1} \ne a'_{k+1}$, then $\mathsf d(\aa, \aa') \ge c$.  Without loss of generatity, it suffices to assume $\gcd(\aa, \aa') = 0$ and $a_{k+1} > 0$.  Since 
$$a_{k+1}b + c(a_1n_1 + \cdots + a_kn_k) = c(a_1'n_1 + \cdots + a_k'n_k)$$
and $\gcd(b,c) = 1$, we must have $c \mid a_{k+1}$.  This implies
$$\mathsf d(\aa,\aa') \ge a_{k+1} \ge c,$$
which completes the proof.  
\end{proof}

We conclude this section with several examples demonstrating that no hypotheses in Theorem~\ref{t:gluing} can be safely omitted.  

\begin{example}\label{e:gluingsmallc}
The numerical monoid $T = \<6, 9, 10, 14\> = 2S + 9\NN$ is a gluing, where $S = \<3, 5, 7\>$.  In this case, $\mathsf c(T) = 3$, even though $\mathsf c(S) = 4$.  
\end{example}

\begin{example}\label{e:gluingmoregens}
The numerical monoid $T = \<15,25,35,18,27\> = 5S_1 + 9S_2$ is a gluing, where $S_1 = \<3,5,7\>$ and $S_2 = \<2,3\>$.  In this case, $\mathsf c(T) = 3$, even though $\mathsf c(S_1) = 4$ and both scaling factors $5$ and $9$ are strictly larger than both $\mathsf c(S_1)$ and $\mathsf c(S_2)$.  
\end{example}

\section{Realizable sets of catenary degrees}
\label{sec:realization}

In this section, we apply Theorem~\ref{t:gluing} to characterize which finite subsets of $\ZZ_{\ge 0}$ are realized as the set of catenary degrees of a numerical monoid, thus providing an alternative answer to \cite[Problem~4.1]{mincatdeg} from that appearing in \cite{catenarycartesian}.  

\begin{example}\label{e:realizationplot}
Theorem~\ref{t:realization} implies $S = \<90, 91, 96, 120, 150\>$ has set of catenary degrees $\{0,2,3,5,6\}$.  Figure~\ref{f:realizationplot} depicts the catenary degrees of the elements of $S$.  The~Betti element $480 \in \Betti(S)$ has catenary degree $\mathsf c(480) = 5$, and the elements $n \in S$ with catenary degree $\mathsf c(n) \ge 5$ are precisely those divisible by $480$ (in the monoid-theoretic sense, i.e.\ $n - 480 \in S$).  
\end{example}

\begin{figure}
\begin{center}
\includegraphics[width=6in]{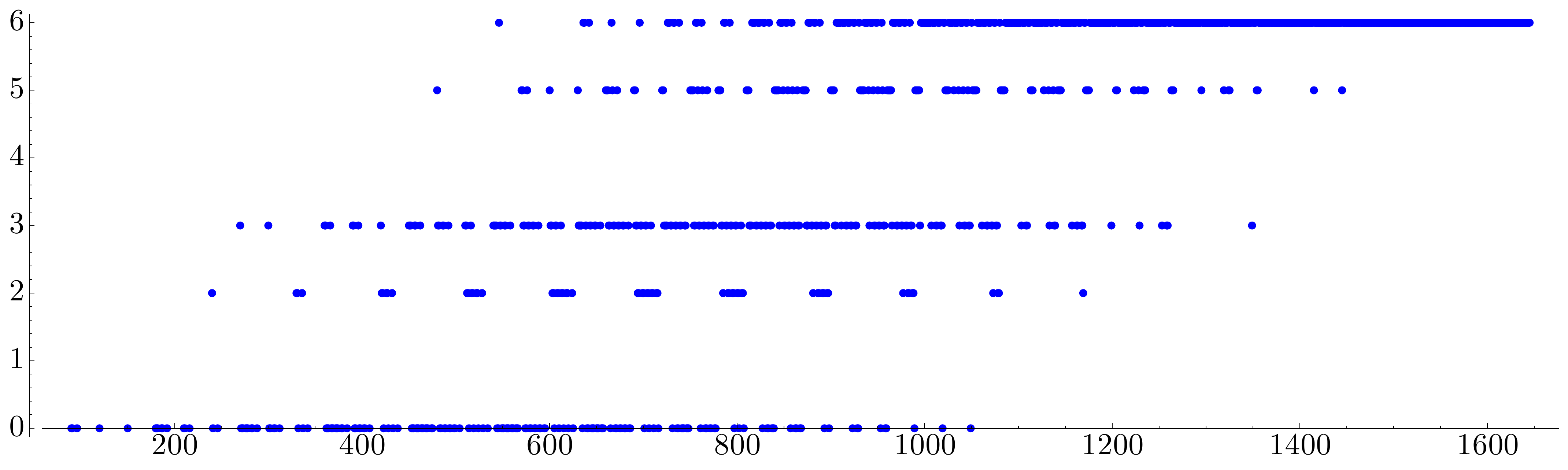}
\end{center}
\caption{The catenary degrees of elements in the monoid from Example~\ref{e:realizationplot}.
}
\label{f:realizationplot}
\end{figure}

\begin{thm}\label{t:realization}
Fix a finite set $C \subset \ZZ_{\ge 0}$.  Then there exists a numerical monoid with set of catenary degrees $C$ if and only if (i) $0 \in C$, (ii) $1 \notin C$, and (iii) $c = \max C \ge 3$.  
\end{thm}

\begin{proof}
For the backward direction, (i) and (ii) both clearly follow from Definition~\ref{d:catenarydegree}, and (iii) follows from the fact that no numerical monoid is half-factorial~\cite{halffactorial}.  

For the converse direction, fix a finite set $C$ satisfying conditions (i), (ii) and (iii).   First, if $C = \{0, c\}$ or $C = \{0, 2, c\}$, then $C$ is the set of catenary degrees of a monoid by \cite[Remark~4.2]{mincatdeg} and \cite[Theorem~4.3]{mincatdeg}, respectively.  In all other cases, $C' = C \cap [0, c)$ satisfies  conditions (i), (ii) and (iii) above, so we can inductively assume that $C'$ is the set of catenary degrees of some numerical monoid $S = \<n_1, \ldots, n_k\>$.  Choosing $b \in S$ such that $b > n_k$ and $\gcd(b,c) = 1$ and applying Theorem~\ref{t:gluing} yields a numerical monoid $T$ with set of catenary degrees $C$, as desired.  
\end{proof}

\begin{remark}\label{r:realizationproof}
In the last sentence of the proof of Theorem~\ref{t:realization}, the choice $b > n_k$ simply ensures the monoid $T$ is minimally generated.  
\end{remark}

\begin{example}\label{e:realization}
Let $C = \{0,2,7,20,26,57\}$.  Following the proof of Theorem~\ref{t:realization}, we begin with the monoid $S = \<3,8,13\>$, which has set of catenary degrees $\{0,2,7\}$.  Subsequent choices of $b$ outlined in the following table yields a numerical monoid with set of catenary degrees $C$.  
\begin{center}
\begin{tabular}{|l|l|l|l|}
$c$ & $b$ & $S$ & Catenary degrees \\
\hline
$7$ &   & $\<3,8,13\>$ & $\{0,2,7\}$ \\
$20$ & $51$ & $\<51,60,160,260\>$ & $\{0,2,7,20\}$ \\
$26$ & $1301$ & $\<1301,1326,1560,4160,6760\>$ & $\{0,2,7,20,26\}$ \\
$57$ & $57001$ & $\<57001,74157,75582,88920,237120,385320\>$ & $\{0,2,7,20,26,57\}$
\end{tabular}
\end{center}
Note that each choice of $b$ lies in the monoid $S$ from the previous row, as required by Theorem~\ref{t:realization} (for instance, $1301 = 11 \cdot 51 + 7 \cdot 60 + 2 \cdot 160$).  This can be readily checked with the \texttt{GAP} package \texttt{numericalsgps} \cite{numericalsgps}.  
\end{example}




\end{document}